\documentclass[11pt,reqno]{article}
\usepackage[utf8]{inputenc}
\usepackage[english]{babel} 

\usepackage{amsmath}
\usepackage{amssymb}
\usepackage{amsthm}
\usepackage{amscd}
\usepackage[hyper]{amsbib}

\usepackage[matrix,arrow,curve]{xy}

\usepackage{stmaryrd}

\begin{document}
\righthyphenmin=2

\renewcommand{\refname}{References}
\renewcommand{\bibname}{Bibliography}
\renewcommand{\proofname}{Proof}

\newtheorem{lm}{Lemma}
\newtheorem{tm}{Theorem}
\newtheorem*{tm*}{Theorem}
\newtheorem*{prop}{Proposition}
\newtheorem*{mtm}{Main Theorem}
\newtheorem*{atm}{Another Presentation}
\newtheorem{cl}{Corollary}
\newtheorem{mcl}{Corollary}
\newtheorem*{cl*}{Corollary}
\theoremstyle{definition}
\newtheorem*{df}{Definition}
\theoremstyle{remark}
\newtheorem*{rk}{Remark}

\newcommand{\Card}{\mathop{\mathrm{Card}}\nolimits}
\newcommand{\Ker}{\mathop{\mathrm{Ker}}\nolimits}
\newcommand{\Cent}{\mathop{\mathrm{Cent}}\nolimits}
\newcommand{\E}{{\mathrm{E}}}
\newcommand{\St}{\mathop{\mathrm{St}}\nolimits}
\newcommand{\Sp}{\mathop{\mathrm{Sp}}\nolimits}
\newcommand{\Ep}{\mathop{\mathrm{Ep}}\nolimits}
\newcommand{\GL}{\mathop{\mathrm{GL}}\nolimits}
\newcommand{\Kt}{\mathop{\mathrm{K_2}}\nolimits}
\newcommand{\Ko}{\mathop{\mathrm{K_1}}\nolimits}
\newcommand{\Ho}{\mathop{\mathrm{H_1}}\nolimits}
\newcommand{\Ht}{\mathop{\mathrm{H_2}}\nolimits}
\newcommand{\epi}{\twoheadrightarrow}
\newcommand{\sgn}{\mathrm{sgn}}
\newcommand{\eps}[1]{\varepsilon_{#1}}
\newcommand{\lan}{\langle}
\newcommand{\ran}{\rangle}
\newcommand{\inv}{^{-1}}
\newcommand{\ur}[1]{\!\,^{(#1)}U_1}
\newcommand{\ps}[1]{\!\,^{(#1)}\!P_1}
\newcommand{\ls}[1]{\!\,^{(#1)}\!L_1}

\renewcommand{\labelenumi}{\theenumi{\rm)}}
\renewcommand{\theenumi}{\alph{enumi}}

\author{Andrei Lavrenov\footnote{The author acknowledges support of the RSCF project 14-11-00297 ``Decomposition of unipotents in reductive groups''.}}

\title{Relative symplectic Steinberg group}

\date{}

\maketitle

\begin{abstract}
We give two definitions of relative symplectic Steinberg group and show that they coincide.
\end{abstract}

\section*{Introduction}

The main result of the present paper is a ``relative version'' of the {\it another presentation} for symplectic Steinberg groups obtained in~\cite{L3}.

First, we give a definition of a relative symplectic group $\St\!\Sp_{2l}(R,\,I,\,\Gamma)$ for any form ideal $(I,\,\Gamma)$ in $R$. For maximal form parameter $\Gamma=I$ it has already appeared in~\cite{SCH-thesis} an is shown to have good homological properties. We also consider case of a maximal form parameter and establish another ``coordinate-free'' presentation for $\St\!\Sp_{2l}(R,\,I,\,\Gamma_\mathrm{max})=\St\!\Sp_{2l}(R,\,I)$. The last result is a relativisation of the~\cite{L3}, but it is not a generalisation, since we need more relations for the presentation. Moreover, in our proofs we use the results of~\cite{L3}.

Namely, the main theorem of~\cite{L3} states the following.

\begin{tm*}
Let $l\geq3$, then the symplectic Steinberg group $\St\!\Sp(2l,\,R)$ can be defined by the set of generators
\begin{multline*}
\big\{X(u,\,v,\,a)\,\big| u,\,v\in V,\ \text{$u$ is a column of}\\ \text{a symplectic elementary matrix},\ \lan u,\,v\ran=0,\ a\in R\big\}
\end{multline*}
and relations
\setcounter{equation}{0}
\renewcommand{\theequation}{P\arabic{equation}}
\begin{align}
&X(u,\,v,\,a)X(u,\,w,\,b)=X(u,\,v+w,\,a+b+\lan v,\,w\ran),\\
&\begin{aligned}X(u,\,va,\,0)=X(v,\,ua,\,0)\,\text{ where}&\text{ $v$ is also a column}\\ &\text{of a symplectic elementary matrix,}\end{aligned}\\
&\begin{aligned}X(u',\,v',\,b)X(u,\,v,\,a)X(u',\,v',\,b&)\inv=\\&=X(T(u',\,v',\,b)u,\,T(u',\,v',\,b)v,\,a),\end{aligned}
\end{align}
where $T(u,\,v,\,a)$ is an ESD-transformation $$w\mapsto w+u(\lan v,\,w\ran+a\lan u,\,w\ran)+v\lan u,\,w\ran.$$ For usual generators of the symplectic Steinberg group the following identities hold 
$$
X_{ij}(a)=X(e_i,\,e_{-j}a\eps{-j},\,0)\ \text{for $j\neq-i$},\qquad X_{i,-i}(a)=X(e_i,\,0,\,a).
$$
\end{tm*}

We prove the following theorem in the case of maximal form parameter.

\begin{tm}
Let $l\geq3$, then the relative symplectic Steinberg group $\St\!\Sp_{2l}(R,\,I)$ can be defined by the set of generators
\begin{multline*}
\big\{(u,\,v,\,a,\,b)\,\big| u,\,v\in V,\ \text{$u$ is a column of}\\ \text{a symplectic elementary matrix},\ \lan u,\,v\ran=0,\ a,\,b\in I\big\}
\end{multline*}
and relations
\setcounter{equation}{0}
\renewcommand{\theequation}{T\arabic{equation}}
\begin{align}
&(u,\,vr,\,a,\,b)=(u,\,v,\,ar,\,b)\text{ for any $r\in R$},\\
&(u,\,v,\,a,\,b)(u,\,w,\,a,\,c)=(u,\,v+w,\,a,\,b+c+a^2\lan v,\,w\ran),\\
&(u,\,v,\,a,\,0)(u,\,v,\,b,\,0)=(u,\,v,\,a+b,\,0),\\
&\begin{aligned}(u,\,v,\,a,\,0)=(v,\,u,\,a,\,0)\,\text{ for $v$ a co}&\text{lumn of}\\&\text{a symplectic elementary matrix,}\end{aligned}\\
&\begin{aligned}(u',\,v',\,a',\,b')(u,\,v,\,a,\,b)(u'&,\,v',\,a',\,b')\inv=\\&=(T(u',\,v'a',\,b')u,\,T(u',\,v'a',\,b')v,\,a,\,b),\end{aligned}\\
&(u,\,u,\,a,\,0)=(u,\,0,\,0,\,2a),\\
&\begin{aligned}(u+vr,\,0,\,0,\,a)=(u,\,0,\,0,\,a)(v,\,0,\,0,\,ar^2)(v,\,u,\,ar,\,0)\text{ for $v$, $u+vr$}&\\ \text{also columns of symplectic elementary matrices}.&\end{aligned}
\end{align}
and for usual relative generators the following identities hold 
$$
Y_{ij}(a)=(e_i,\,e_{-j},\,a\eps{-j},\,0)\ \text{for $j\neq-i$},\qquad Y_{i,-i}(a)=(e_i,\,0,\,0,\,a).
$$

\end{tm}

The author plans to generalise the above result for the case of arbitrary form parameter.

\section{Relative symplectic Steinberg group}

In the sequel $R$ denotes an arbitrary associative commutative unital ring, $V=R^{2l}$ denotes a free right $R$-module with basis numbered $e_{-l}$, $\ldots$, $e_{-1}$, $e_1$, $\ldots$, $e_l$, $l\geq3$. For the vector $v\in V$ its $i$-th coordinate will be denoted by $v_i$, i.e. $v=\sum_{i=-l}^le_iv_i$. By $\lan\ ,\ \ran$ we denote the standard symplectic form on $V$, i.e $\lan e_i,\,e_j\ran=\sgn(i)\delta_{i,-j}$. We will usually write $\eps i$ instead of $\sgn(i)$. Observe that $\lan u,\,u\ran=0$ for any $u\in V$.

\begin{df}
Define the {\it symplectic group} $\Sp(V)=\Sp_{2l}(R)$ as the group of automorphisms of $V$ preserving the symplectic form $\langle\ ,\ \rangle$,
$$\Sp(V)=\{f\in\GL(V)\mid\lan f(u),\,f(v)\ran=\lan u,\,v\ran\ \ \forall\,u,\,v\in V\}.$$
\end{df}

\begin{df}[Eichler--Siegel--Dickson transformations]
For $a\in R$ and $u$, $v\in V$, $\lan u,\,v\ran=0$, denote by $T(u,\,v,\,a)$ the automorphism of $V$ s.t. for $w\in V$ one has
$$T(u,\,v,\,a)\colon w\mapsto w+u(\lan v,\,w\ran+a\lan u,\,w\ran)+v\lan u,\,w\ran.$$
We refer to the elements $T(u,\,v,\,a)$ as the ({\it symplectic}) {\it ESD-trans\-for\-ma\-tions}.
\end{df}

\begin{lm}
\label{esd-properties}
Let $u$, $v$, $w\in V$ be three vectors such that $\lan u,\,v\ran=0$, $\lan u,\,w\ran=0$, and let $a$, $b\in R$. Then
\begin{enumerate}
\item $T(u,\,v,\,a)\in\Sp(V)$,
\item $T(u,\,v,\,a)\,T(u,\,w,\,b)=T(u,\,v+w,\,a+b+\lan v,\,w\ran)$,
\item $T(u,\,va,\,0)=T(v,\,ua,\,0)$,
\item $g\,T(u,\,v,\,a)g\inv=T(gu,\,gv,\,a)\ \ \forall\,g\in\Sp(V)$.
\end{enumerate}
\end{lm}

\begin{rk}
Observe that $T(u,\,0,\,0)=1$ and $T(u,\,v,\,a)\inv=T(u,\,-v,\,-a)$.
\end{rk}

In the present paper all commutators are left-normed, $[x,\,y]=xyx\inv y\inv$, we denote $xyx\inv$ by $\!\,^xy$.

\begin{lm}
\label{z-decomposition}
For $u$, $v\in V$ such that $u_i=u_{-i}=v_i=v_{-i}=0$, $\lan u,\,v\ran=0$, and $a\in R$ one has 
$$
[T(e_i,\,u,\,0),\,T(e_{-i},\,v,\,a)]=T(u,\,v\eps{i},\,a)T(e_{-i},\,-ua\eps{-i},\,0).
$$
\end{lm}

\begin{df}
{\it A form ideal} $(I,\,\Gamma)$ is a pair of an ideal $I\trianglelefteq R$ and a {\it relative form parameter} $\Gamma$ {\it of level} $I$, i.e., an additive subgroup of $I$ such that 
\begin{enumerate}
\item $\forall\, a\in I\,$ holds $2a\in\Gamma$,
\item $\forall\, r\in R,\ \forall\, a\in I\,$ holds $ra^2\in\Gamma$,
\item $\forall\,\alpha\in\Gamma,\ \forall\,r\in R\,$ holds $\alpha r^2\in\Gamma$.
\end{enumerate}
\end{df}

\begin{rk}
If $2\in R^\times$ the only possible choice for a relative form parameter is $\Gamma=I$.
\end{rk}

\begin{df}
Define $T_{ij}(a)=T(e_{i},\,e_{-j}a\eps{-j},\,0)$ and $T_{i,-i}(a)=T(e_i,\,0,\,a)$, where $a\in R$, $i$, $j\in\{-l,$ $\ldots,$ $-1,$ $1,$ $\ldots,$ $l\}$, $i\not\in\{\pm j\}$. We refer to these elements as the {\it elementary symplectic transvections}. A normal subgroup of $\Sp(V)$
$$\Ep_{2l}(R,\,I,\,\Gamma)=\!\,^{\Sp(V)}\lan T_{ij}(a),\ T_{i,-i}(\alpha)\mid i\not\in\{\pm j\},\ a\in I,\ \alpha\in\Gamma\ran$$
is called {\it the relative elementary symplectic group} corresponding to the form ideal $(I,\,\Gamma)$.
\end{df}

\begin{df}
The {\it symplectic Steinberg group} $\St\!\Sp_{2l}(R)$ is the group generated by the formal symbols $X_{ij}(r)$, $i\neq j$, $r\in R$ subject to the Steinberg relations
\setcounter{equation}{-1}
\renewcommand{\theequation}{S\arabic{equation}}
\begin{align}
&X_{ij}(r)=X_{-j,-i}(-r\eps i\eps j),\\
&X_{ij}(r)X_{ij}(s)=X_{ij}(r+s),\\
&[X_{ij}(r),\,X_{hk}(s)]=1,\text{ for }h\not\in\{j,-i\},\ k\not\in\{i,-j\},\\
&[X_{ij}(r),\,X_{jk}(s)]=X_{ik}(rs),\text{ for }i\not\in\{-j,-k\},\ j\neq-k,\\
&[X_{i,-i}(r),\,X_{-i,j}(s)]=X_{ij}(rs\eps i)X_{-j,j}(-rs^2),\\
&[X_{ij}(r),\,X_{j,-i}(s)]=X_{i,-i}(2rs\eps i).
\end{align}
\end{df}

The next lemma is a straightforward consequence of Lemmas~\ref{esd-properties} and \ref{z-decomposition}.

\begin{lm}
There is a natural epimorphism 
$$\phi\colon\St\!\Sp_{2l}(R)\epi\Ep_{2l}(R)=\Ep_{2l}(R,\,R,\,R)$$ 
sending the generators $X_{ij}(a)$ to the corresponding elementary transvections $T_{ij}(a)$. In other words, the Steinberg relations hold for the elementary transvections.
\end{lm}

If group $G$ acts on group $H$ from the left, we will denote the image of $h\in H$ under the homomorphism corresponding to the element $g\in G$ by $\!\,^gh$, the element $\!\,^gh\cdot h\inv$ by $\llbracket g,\,h]$ and the element $h\cdot\!\,^gh\inv$ by $[h,\,g\rrbracket$.

\begin{df}
Define the {\it relative symplectic Steinberg group} $\St\!\Sp_{2l}(R,\,I,\,\Gamma)$ corresponding to the form ideal $(I,\,\Gamma)$ as a formal group with the action of the absolute Steinberg group $\St\!\Sp_{2l}(R)$ defined by the set of (relative) generators $\{Y_{ij}(a)\mid i\not\in\{\pm j\},\ a\in I\}\cup\{Y_{i,-i}(\alpha)\mid \alpha\in\Gamma\}$ subject to the following relations
\setcounter{equation}{-1}
\renewcommand{\theequation}{KL\arabic{equation}}
\begin{align}
&Y_{ij}(a)=Y_{-j,-i}(-a\eps i\eps j),\\
&Y_{ij}(a)Y_{ij}(b)=Y_{ij}(a+b),\\
&\llbracket X_{ij}(r),\,Y_{hk}(a)]=1,\text{ for }h\not\in\{j,-i\},\ k\not\in\{i,-j\},\\
&\llbracket X_{ij}(r),\,Y_{jk}(a)]=Y_{ik}(ra),\text{ for }i\not\in\{-j,-k\},\ j\neq-k,\\
&\llbracket X_{i,-i}(r),\,Y_{-i,j}(a)]=Y_{ij}(ra\eps i)Y_{-j,j}(-ra^2),\\
&[Y_{i,-i}(\alpha),\,X_{-i,j}(r)\rrbracket=Y_{ij}(\alpha r\eps i)Y_{-j,j}(-\alpha r^2),\\
&\llbracket X_{ij}(r),\,Y_{j,-i}(a)]=X_{i,-i}(2ra\eps i),\\
&\!\,^{X_{ij}(a)}Y_{hk}(b)=\!\,^{Y_{ij}(a)}Y_{hk}(b).
\end{align}
In other words we consider a free group generated by symbols $(g,\,x)=\!\,^gx$ where $g$ is from the absolute Steinberg group and $x$ is from the set of relative generators, $\St\!\Sp_{2l}(R)$ naturally acts on this free group via $\!\,^f(g,\,x)=(fg,\,x)$ and then we define a relative symplectic Steinberg group as a factor of the described free group modulo normal subgroup generated by KL0--KL7.
\end{df}

\begin{df}
Obviously, there is a natural map $$\varphi:\St\!\Sp_{2l}(R,\,I,\,\Gamma)\rightarrow\Sp_{2l}(R).$$ Then its kernel is denoted by $\Kt\!\Sp_{2l}(R,\,I,\,\Gamma)$.
\end{df}

\section{Auxiliary constructions}

\begin{df}
Define the {\it relative Steinberg unipotent radical} $$\ur i=\lan Y_{ij}(a),\ Y_{i,-i}(\alpha)\mid i\not\in\{\pm j\},\ a\in I,\ \alpha\in\Gamma\ran\leq\St\!\Sp_{2l}(R,\,I,\,\Gamma)$$ and the ({\it absolute}) {\it Steinberg parabolic subgroup} $$\ps i=\lan X_{kh}(a)\mid i\not\in\{h,-k\},\ a\in R\ran\leq\St\!\Sp_{2l}(R).$$
\end{df}

\begin{lm}[Levi decomposition]
For $g\in\ps i$, $u\in\ur i$ one has 
$$\!\,^gu\in\ur i.$$
\end{lm}

\begin{lm}
\label{heis}
One has
$$[\ur i,\,\ur i]\leq\lan Y_{i,-i}(\alpha)\ran,\qquad [\ur i,\,\lan Y_{i,-i}(\alpha)\ran]=1.$$
\end{lm}

\begin{cl*}
Every element of $\ur i$ can be expressed in the form $$Y_{i,-i}(\alpha)Y_{i,-l}(a_{-l})\ldots Y_{i,-1}(a_{-1})Y_{i,1}(a_1)\ldots Y_{i,l}(a_l).$$
\end{cl*}

\begin{lm}
\label{unirad}
The restriction of the natural projection $\varphi\colon\St\!\Sp_{2l}(R,\,I,\,\Gamma)\epi\Sp_{2l}(R)$ to $\ur i$ is injective
$$\ur i\cong \varphi(\ur i).$$
\end{lm}

\begin{proof}
Take an element $x\in\ur i$. Using the above corollary, decompose $x$ as $$x=Y_{i,-i}(a_{-i})Y_{i,-l}(a_{-l})\ldots Y_{i,-1}(a_{-1})Y_{i,1}(a_1)\ldots Y_{i,l}(a_l).$$ Then $\varphi(x)=1$ implies that $a_i=0$ for all $i$.
\end{proof}

\begin{lm}
\label{unipotent-decomposition}
Take $v\in I^{2l}$ such that $v_{-i}=0$ and $\alpha\in\Gamma$. Denote $$v_-=\sum_{k<0}e_kv_k\qquad \text{and}\qquad v_+=\sum_{k>0}e_kv_k.$$ Then
\begin{multline*}
T(e_i,\,v,\,\lan v_-,\,v_+\ran+\alpha)=T_{i,-i}(\alpha+2v_i)\cdot\\ \cdot T_{-l,-i}(v_{-l}\eps i)\ldots T_{-1,-i}(v_{-1}\eps i)T_{1,-i}(v_1\eps i)\ldots T_{l,-i}(v_l\eps i).
\end{multline*}
\end{lm}

\begin{proof}
Assume that $i>0$, for $i<0$ the proof looks exactly the same. Since $T_{j,-i}(v_j\eps i)=T(e_i,\,e_jv_j,\,0)$ for $j\neq-i$ and $$T_{i,-i}(2v_i)=T(e_i,\,0,\,2v_i)=T(e_i,\,e_iv_i,\,0),$$ one has 
\begin{multline*}
T_{-l,-i}(v_{-l}\eps i)\ldots T_{-1,-i}(v_{-1}\eps i)=\\=T(e_i,\,e_{-l}v_{-l},\,0)\ldots T(e_i,\,e_{-1}v_{-1},\,0)=T(e_i,\,v_-,\,0),
\end{multline*}
and
\begin{multline*}
T_{i,-i}(2v_i)T_{1,-i}(v_{1}\eps i)\ldots T_{l,-i}(v_{l}\eps i)=\\=T(e_i,\,e_iv_i,\,0)T(e_i,\,e_{1}v_{1},\,0)\ldots T(e_i,\,e_{l}v_{l},\,0)=T(e_i,\,v_+,\,0),
\end{multline*}
so that the right hand side of the desired equality is in fact equal to
\begin{multline*}
T_{i,-i}(\alpha)T(e_i,\,v_-,\,0)T(e_i,\,v_+,\,0)=\\=T(e_i,\,0,\,\alpha)T(e_i,v,\lan v_-,\,v_+\ran)=T(e_i,\,v,\,\lan v_-,\,v_+\ran+\alpha).
\end{multline*}
\end{proof}

\begin{df}
For $v\in I^{2l}$ with $v_{-i}=0$ and $a\in R$ such that $a-\lan v_-,\,v_+\ran\in\Gamma$, define $$Y(e_i,\,v,\,a)=(\phi\vert_{\ur i})\inv\big(T(e_i,\,v,\,a)\big).$$
\end{df}

\begin{rk}
By Lemma~\ref{unipotent-decomposition}, $T(e_i,\,v,\,a)$ indeed lies in $\phi(\ur i)$. Moreover, the same lemma provides the following decomposition.
\end{rk}

\begin{lm}
\label{y-decomposition}
For $v\in I^{2l}$ such that $v_{-i}=0$, $a\in R$ such that $a-\lan v_-,\,v_+\ran\in\Gamma$ one has
\begin{multline*}
Y(e_i,\,v,\,a)=Y_{i,-i}(a+2v_i-\lan v_-,\,v_+\ran)\cdot\\ \cdot Y_{-l,-i}(v_{-l}\eps i)\ldots Y_{-1,-i}(v_{-1}\eps i)Y_{1,-i}(v_1\eps i)\ldots Y_{l,-i}(v_l\eps i).
\end{multline*}
\end{lm}

\setcounter{cl}{0}

\begin{cl*}
In particular, $Y(e_{-j},\,-e_ia\eps j,\,0)=Y_{ij}(a)$ for $i\not\in\{\pm j\}$ and $Y(e_i,\,0,\,\alpha)=Y_{i,-i}(\alpha)$.
\end{cl*}

\begin{lm}
\label{y-add}
For $v$, $w\in I^{2l}$ such that $v_{-i}=w_{-i}=0$ and $a$, $b\in R$ such that $a-\lan v_-,\,v_+\ran$, $b-\lan w_-,\,w_+\ran\in\Gamma$, one has 
$$
Y(e_i,\,v,\,a)Y(e_i,\,w,\,b)=Y(e_i,\,v+w,\,a+b+\lan v,\,w\ran).
$$
\end{lm}

\begin{proof}
Obviously, $(v+w)_{-i}=0$. Moreover, 
$$
a+b+\lan v,\,w\ran-\lan (v+w)_-,\,(v+w)_+\ran=a-\lan v_-,\,v_+\ran+b-\lan w_-,\,w_+\ran\in\Gamma,
$$
so that the right hand side of this equality is well-defined. Now, it remains to observe that the images of the elements on both sides under $\varphi$ coincide.
\end{proof}

\begin{cl*}
One has $Y(e_i,\,0,\,0)=1$ and $\,Y(e_i,\,v,\,a)\inv=Y(e_i,\,-v,\,-a)$.
\end{cl*}

\begin{lm}
\label{wow}
For $f\in\Ep_{2l}(R)$, $v\in I^{2l}$ one has 
$$\lan(fv)_-,\,(fv)_+\ran-\lan v_-,\,v_+\ran\in\Gamma.$$
\end{lm}

\begin{proof}
We may assume that $f$ is an elementary transvection. For short root transvection one has
\begin{multline*}
\lan(T_{ij}(r)v)_-,\,(T_{ij}(r)v)_+\ran-\lan v_-,\,v_+\ran=\\=-(v_i+v_jr)v_{-i}-v_j(v_{-j}-v_{-i}r\eps i\eps j)+v_iv_{-i}+v_jv_{-j}=\\=v_{-i}v_jr(\eps i\eps j-1)\in2I\subseteq\Gamma.
\end{multline*}
For long root transvection one has
\begin{multline*}
\lan(T_{i,-i}(r)v)_-,\,(T_{i,-i}(r)v)_+\ran-\lan v_-,\,v_+\ran=\\=-(v_i+v_{-i}r\eps i)v_{-i}+v_iv_{-i}=-\eps irv_{-i}^2\in\Gamma.
\end{multline*}
\end{proof}

\begin{lm}
\label{y-conjugated-by-ps}
For $g\in\ps i$, $v\in I^{2l}$ such that $v_{-i}=0$, $a\in R$ such that $a-\lan v_-,\,v_+\ran\in\Gamma$, one has 
$$\!\,^g\,Y(e_i,\,v,\,a)=Y(e_i,\,\phi(g)v,\,a).$$
\end{lm}

\begin{proof}
First, observe that since $T_{kh}(a)e_i=e_i$ for $i\not\in\{h,-k\}$ one has $\phi(g)e_i=e_i$. Thus, $$\lan e_i,\,\phi(g)v\ran=\lan \phi(g)e_i,\,\phi(g)v\ran=\lan e_i,\,v\ran=0,$$ i.e., $(\phi(g)v)_{-i}=0$. By Lemma~\ref{wow}, $a-\lan(\phi(g)v)_-,\,(\phi(g)v)_+\ran\in\Gamma$, so the right hand side of the desired equation is well-defined. Finally, observe that the images of both sides under $\varphi$ coincide.
\end{proof}

\begin{lm}
\label{switch}
For $j\neq-i$, $a\in I$, one has $Y(e_i,\,e_ja,\,0)=Y(e_j,\,e_ia,\,0)$.
\end{lm}

\begin{proof}
For $i=j$ the claim is obvious. Let $i\neq j$, then
$$Y(e_i,\,e_ja,\,0)=Y_{-j,i}(a\eps i)=Y_{-i,j}(-a\eps j)=Y(e_j,\,e_ia,\,0).$$
\end{proof}

\begin{rk}
In the absolute situation $(I,\,\Gamma)=(R,\,R)$, we will write $X(e_i,\,v,\,a)$ instead of $Y(e_i,\,v,\,a)$. These are exactly the elements, which appear in the {\it another presentation}.
\end{rk}

\begin{lm}
\label{z-decomposition-for-y}
For $v\in V$ such that $v_{-j}=v_{k}=v_{-k}=0$, $k\not\in\{\pm j\}$, $a\in R$, $b\in I$ one has
$$
[Y(e_k,\,e_jb,\,0),\,X(e_{-k},\,v,\,a)\rrbracket=Y(e_j,\,vb\eps{k},\,ab^2)Y(e_{-k},\,-e_jab\eps{-k},\,0).
$$
\end{lm}

\begin{proof}
The proof of Lemma~12 from the Another presentation paper can be repeated verbatim.
\end{proof}

\begin{cl*}
For $v\in V$ such that $v_{-j}=v_{k}=v_{-k}=0$, $k\not\in\{\pm j\}$, $a\in R$, $b\in I$ one has the following decomposition
$$
Y(e_j,\,vb,\,ab^2)=[Y(e_k,\,e_jb,\,0),\,X(e_{-k},\,v\eps{k},\,a)\rrbracket Y(e_{-k},\,e_jab\eps{-k},\,0).
$$
\end{cl*}

\begin{lm}
\label{ppc}
For $j\not\in\{\pm k\}$, $v\in I^{2l}$ such that $v_{-j}=v_{k}=v_{-k}=0$, $r\in R$, $a$ such that $a-\lan v_-,\,v_+\ran\in\Gamma$, one has
$$
\llbracket X(e_k,\,e_jr,\,0),\,Y(e_{-k},\,v,\,a)]=Y(e_j,\,vr\eps k,\,ar^2)Y(e_{-k},\,e_jar\eps k,\,0).
$$
\end{lm}

\begin{proof}
Denote $x=X_{j,-k}(r\eps{k})=X(e_k,\,e_jr,\,0)$, $w=Y_{-k,k}(a-\lan v_-,\,v_+\ran)$, $y=Y_{-l,k}(v_{-l}\eps{-k})\cdot\ldots\cdot Y_{-1,k}(v_{-1}\eps{-k})$, $z=Y_{1,k}(v_1\eps{-k})\cdot\ldots\cdot Y_{l,k}(v_l\eps{-k})$. Then
$$
\llbracket X(e_k,\,e_jr,\,0),\,Y(e_{-k},\,v,\,a)]=\llbracket x,\,yzw]=\llbracket x,\,y]\cdot\!\,^y\llbracket x,\,z]\cdot\!\,^{yz}\llbracket x,\,w].
$$
Assume that $j<0$, for $j>0$ the proof looks exactly the same. For $h\not\in\{\pm j,\pm k\}$ one has
\begin{multline*}
\llbracket X_{j,-k}(r\eps k),\,Y_{h,k}(v_h\eps{-k})]=\llbracket X_{j,-k}(r\eps k),\,Y_{-k,-h}(v_h\eps{h})]=\\
=Y_{j,-h}(rv_h\eps k\eps h)=Y_{h,-j}(rv_h\eps k\eps j).
\end{multline*}
For $h=j$ one has 
$$
\llbracket X_{j,-k}(r\eps k),\,Y_{j,k}(v_j\eps{-k})]=\llbracket X_{j,-k}(r\eps k),\,Y_{-k,-j}(v_j\eps{j})]=Y_{j,-j}(2rv_j\eps k).
$$
Since any $Y_{h,-j}(\hat a)$ commutes with any $Y_{tk}(\hat b)$ for $h$, $t\not\in\{\pm k, -j\}$, $h$, $t<0$, we have
\begin{multline*}
\llbracket x,\,y]=\prod_{h<0}\llbracket X_{j,-k}(r\eps k),\,Y_{h,k}(v_h\eps{-k})]=\\=Y_{j,-j}(2rv_j\eps k)\prod_{\substack{h\neq j\\h<0}}Y_{h,-j}(rv_h\eps k\eps j)=Y(e_j,\,v_-r\eps k,\,0).
\end{multline*}
Similarly,
$$
\llbracket x,\,z]=\prod_{\substack{h>0}}Y_{h,-j}(rv_h\eps k\eps j)=Y(e_j,\,v_+r\eps k,\,0).
$$
Then,
\begin{multline*}
\!\,^y\llbracket x,\,z]=\!\,^{Y(e_{-k},\,v_-,\,0)}Y(e_j,\,v_+r\eps k,\,0)=\\
=Y(e_j,\,v_+r\eps k,\,0)[Y(e_j,\,-v_+r\eps k,\,0),\,Y(e_{-k},\,v_-,\,0)]=\\
=Y(e_j,\,v_+r\eps k,\,0)\llbracket X(e_j,\,-v_+r\eps k,\,0),\,Y(e_{-k},\,v_-,\,0)]=\\
=Y(e_j,\,v_+r\eps k,\,0)Y(e_{-k},\,e_j\lan v_-,\,v_+\ran r\eps k,\,0).
\end{multline*}
Next,
\begin{multline*}
\llbracket x,\,w]=\llbracket X_{j,-k}(r\eps{k}),\,Y_{-k,k}(a-\lan v_-,\,v_+\ran)]=\\
=\Big([Y_{-k,k}(a-\lan v_-,\,v_+\ran),\,X_{k,-j}(r\eps j)\rrbracket\Big)\inv=\\
=\Big(Y_{-k,-j}\big((a-\lan v_-,\,v_+\ran)r\eps j\eps{-k}\big)Y_{j,-j}\big(-(a-\lan v_-,\,v_+\ran)r^2\big)\Big)\inv=\\
=Y_{j,-j}(ar^2-\lan v_-r\eps k,\,v_+r\eps k\ran)Y(e_j,\,e_{-k}(a-\lan v_-,\,v_+\ran)r\eps{k},\,0).
\end{multline*}
Obviously, $\llbracket x,\,w]$ commutes with $y$ and $z$, thus, finally,
\begin{multline*}
\llbracket x,\,yzw]=Y_{j,-j}(ar^2-\lan v_-r\eps k,\,v_+r\eps k\ran)Y_{j,-j}(2rv_j\eps k)\cdot\\
\cdot\prod_{\substack{h\neq j\\h<0}}Y_{h,-j}(rv_h\eps k\eps j)\prod_{\substack{h>0}}Y_{h,-j}(rv_h\eps k\eps j)Y(e_{-k},\,e_j\lan v_-,\,v_+\ran r\eps k,\,0)\cdot\\
\cdot Y(e_{-k},\,e_{j}(a-\lan v_-,\,v_+\ran)r\eps{k},\,0)=Y(e_j,\,vr\eps k,\,ar^2)Y(e_{-k},\,e_jar\eps k,\,0).
\end{multline*}
\end{proof}

\begin{cl*}
For $j\not\in\{\pm k\}$, $v\in I^{2l}$ such that $v_{-j}=v_{k}=v_{-k}=0$, $r\in R$, $a$ such that $a-\lan v_-,\,v_+\ran\in\Gamma$, one has
$$
Y(e_j,\,vr,\,ar^2)=\llbracket X(e_k,\,e_jr,\,0),\,Y(e_{-k},\,v\eps k,\,a)]Y(e_{-k},\,e_jar\eps{-k},\,0).
$$
\end{cl*}

\begin{df}
For $u\in V$, $v\in I^{2l}$ such that $\lan u,\,v\ran=0$, $u_i=u_{-i}=v_i=v_{-i}=0$, $a\in R$ such that $a-\lan v_-,\,v_+\ran\in\Gamma$ denote
$$
Y_{(i)}(u,\,v,\,a)=\llbracket X(e_i,\,u,\,0),\,Y(e_{-i},\,v\eps{i},\,a)]Y(e_{-i},\,ua\eps{-i},\,0).
$$
\end{df}

\begin{rk}
Due to Lemma~\ref{z-decomposition} one has $\varphi\big(Y_{(i)}(u,\,v,\,a)\big)=T(u,\,v,\,a)$.
\end{rk}

\begin{rk}
For $v\in I^{2l}$ such that $v_{-j}=v_i=v_{-i}=0$, $a\in R$ such that $a-\lan v_-,\,v_+\ran\in\Gamma$, one has
$$
Y_{(i)}(e_j,\,v,\,a)=Y(e_j,\,v,\,a)
$$
by Lemma~\ref{ppc}. One can also obtain the following result.
\end{rk}

\begin{lm}
For $v\in V$ with $v_{-j}=v_i=v_{-i}=0$, $b\in I$ one has
$$
Y_{(i)}(v,\,e_jb,\,0)=Y(e_j,\,vb,\,0).
$$
\end{lm}

\begin{proof}
The proof of Lemma~15 from the Another presentation paper can be repeated verbatim.
\end{proof}

\begin{lm}
\label{z-correctness}
Consider $u\in V$, $v\in I^{2l}$ such that $\lan u,\,v\ran=0$, $u_i=u_{-i}=u_{j}=u_{-j}=0$, $v_i=v_{-i}=v_j=v_{-j}=0$, $r\in R$ and $a\in R$ such that $a-\lan v_-,\,v_+\ran\in\Gamma$. Then one has
$$
Y_{(i)}(u,\,vr,\,ar^2)=Y_{(j)}(ur,\,v,\,a).
$$
\end{lm}

\begin{proof}
Denote $x=X(e_i,\,u,\,0)$, $q=Y(e_{-i},\,vr\eps i,\,ar^2)$, $c=Y(e_{-i},\,uar^2\eps{-i},\,0)$, then by the very definition
$$
Y_{(i)}(u,\,vr,\,ar^2)=\llbracket x,\,q]c.
$$
Denote $y=X(e_j,\,e_{-i}r\eps{-i},\,0)$, $z=Y(e_{-j},\,v\eps j,\,a)$, $w=Y(e_{-j},\,e_{-i}ar\eps i\eps{-j},\,0)$, then by Lemma~\ref{ppc}
$$
q=\llbracket y\inv,\,z]w
$$
and
$$
\llbracket x,\,q]=\llbracket x,\,\llbracket y\inv,\,z]w]=\llbracket x,\,\llbracket y\inv,\,z]]\cdot\!\,^{\llbracket y\inv,\,z]}\llbracket x,\,w].
$$
Using Hall--Witt identity, 
$$
\llbracket x,\,\llbracket y\inv,\,z]]=\!\,^{y\inv x}\llbracket[x\inv,\,y],\,z]\cdot\!\,^{y\inv z}[[z\inv,\,x\rrbracket,\,y\rrbracket.
$$
One can check using Lemma~\ref{y-conjugated-by-ps} that $x$ acts trivially on $z$, i.e., $[z\inv,\,x\rrbracket=1$, thus,
$$
\llbracket x,\,\llbracket y\inv,\,z]]=\!\,^{y\inv x}\llbracket[x\inv,\,y],\,z].
$$
Next, 
$$
[x\inv,\,y]=[X(e_i,\,-u,\,0),\,X(e_j,\,-e_{-i}r\eps i,\,0)]=X(-u,\,-e_jr,\,0)=X(e_j,\,ur,\,0).
$$
Denote $d=[x\inv,\,y]=X(e_j,\,ur,\,0)$, one can show that $[x,\,d]=1$, thus,
$$
\llbracket x,\,\llbracket y\inv,\,z]]=\!\,^{y\inv}\llbracket d,\,z].
$$
Observe that $y$ acts trivially on $q$, thus, $\!\,^y\llbracket x,\,q]=\llbracket yx,\,q]$. Since $yx=xy[y\inv,\,x\inv]$ and 
$$
[y\inv,\,x\inv]=[X(e_{-i},\,e_jr\eps i,\,0),\,X(e_i,\,-u,\,0)]=X(e_j,\,ur,\,0)
$$
also acts trivially on $q$, one has $\!\,^y\llbracket x,\,q]=\llbracket x,\,q]$, thus,
$$
\llbracket x,\,q]=\!\,^y\llbracket x,\,q]=\!\,^y\llbracket x,\,\llbracket y\inv,\,z]]\cdot\!\,^{y\llbracket y\inv,\,z]}\llbracket x,\,w]=\llbracket d,\,z]\cdot\!\,^{y\llbracket y\inv,\,z]}\llbracket x,\,w].
$$
Denote 
$$h=\llbracket x,\,w]=\llbracket X(e_i,\,u,\,0),\,Y(e_{-i},\,e_{-j}ar\eps i\eps{-j},\,0)]=Y(e_{-j},\,uar\eps{-j},\,0),$$
then $[w,\,h]=1$, thus, $\!\,^{\llbracket y\inv,\,z]}h=\!\,^{\llbracket y\inv,\,z]w}h=\!\,^{q}h=h$, so,
$$
\llbracket x,\,q]=\llbracket d,\,z]\cdot\!\,^{y}h=\llbracket d,\,z]\cdot h\cdot[h\inv,\,y\rrbracket.
$$
Since
$$
\llbracket y,\,h\inv]=\llbracket X(e_j,\,-e_{-i}r\eps i,\,0),\,Y(e_{-j},\,uar\eps j,\,0)]=Y(e_{-i},\,uar^2\eps{-i},\,0)=c,
$$
one has
$$
\llbracket x,\,q]c=\llbracket d,\,z]h
$$
or,
$$
Y_{(i)}(u,\,vr,\,ar^2)=Y_{(j)}(ur,\,v,\,a).
$$
\end{proof}

\begin{df}
Define the ({\it absolute}) {\it Levi subgroup} $\ls i=\ps i\cap\ps{-i}$.
\end{df}

\begin{rk}
Observe that for $g\in\ls i$ one has $\phi(g)e_i=e_i$ and $\phi(g)e_{-i}=e_{-i}$. Indeed, the first equality holds for $g=X_{kh}(a)$ with $\{-k,h\}\not\ni i$ and the second one for $g=X_{kh}(a)$ with $\{-k,h\}\not\ni -i$.
\end{rk}

\begin{lm}
For $u\in V$, $v\in I^{2l}$ such that $\lan u,\,v\ran=0$, $u_i=u_{-i}=v_i=v_{-i}=0$, $a\in R$ such that $a-\lan v_-,\,v_+\ran\in\Gamma$, $g\in\ls i$, one has
$$
g\,Y_{(i)}(u,\,v,\,a)g\inv=Y_{(i)}(\phi(g)u,\,\phi(g)v,\,a).
$$
\end{lm}

\begin{rk}
Since $\lan\phi(g)u,\,e_i\ran=\lan\phi(g)u,\,\phi(g)e_i\ran=\lan u,\,e_i\ran=0$, one can see that $(\phi(g)u)_{-i}=0$ and similarly $(\phi(g)u)_{i}=(\phi(g)v)_{-i}=(\phi(g)v)_{i}=0$. Using Lemma~\ref{wow} one obtains, that $a-\lan(\phi(g)v)_-,\,(\phi(g)v)_+\ran\in\Gamma$. Thus, $Y_{(i)}(\phi(g)u,\,\phi(g)v,\,a)$ is well-defined.
\end{rk}

\begin{proof}
The proof of Lemma~17 from the Another presentation paper works for this situation as well.
\end{proof}

\begin{rk}
Lemma~\ref{unipotent-decomposition} implies that for $v$ such that $v_{-i}=v_j=v_{-j}=0$, $j\not\in\{\pm i\}$ one has $Y(e_i,\,v,\,a)\in\ls j$. 
\end{rk}

\begin{rk}
For $w$ orthogonal to both $u$ and $v$, $\lan u,\,w\ran=\lan v,\,w\ran=0$, one has $T(u,\,v,\,a)w=w$. Below, in the computations this fact is frequently used without any special reference. 
\end{rk}

\begin{lm}
\label{z-additivity}
For $j\not\in\{\pm i\}$ and $u\in V$ such that $u_i=u_{-i}=u_j=u_{-j}=0$, $v\in I^{2l}$ such that $v_i=v_{-i}=0$ and $\lan u,\,v\ran=0$, and for $a$, $b\in I$ such that $a-\lan v_-,\,v_+\ran\in\Gamma$, one has
$$
Y_{(i)}(u,\,v,\,a)Y_{(i)}(u,\,e_jb,\,0)=Y_{(i)}(u,\,v+e_jb,\,a+v_{-j}b\eps{-j}).
$$
\end{lm}

\begin{proof}
Start with the right-hand side (one can check that it is well-defined)
\begin{multline*}
Y_{(i)}(u,\,v+e_jb,\,a+v_{-j}b\eps j)=\\
=\llbracket X(e_i,\,u,\,0),\,Y(e_{-i},\,(v+e_jb)\eps{i},\,a+v_{-j}b\eps j)]Y(e_{-i},\,u(a+v_{-j}b\eps{-j})\eps{-i},\,0).
\end{multline*}
Decompose $Y(e_{-i},\,(v+e_jb)\eps{i},\,a+v_{-j}b\eps j)$ inside the commutator and use the familiar identity $\llbracket a,\,bc]=\llbracket a,\,b]\cdot\,^b\llbracket a,\,c]$ to obtain
\begin{multline*}
\llbracket X(e_i,\,u,\,0),\,Y(e_{-i},\,(v+e_jb)\eps{i},\,a+v_{-j}b\eps j)]=\\
=\llbracket X(e_i,\,u,\,0),\,Y(e_{-i},\,v\eps{i},\,a)Y(e_{-i},\,e_jb\eps{i},\,0)]=\\
=\llbracket X(e_i,\,u,\,0),\,Y(e_{-i},\,v\eps{i},\,a)]\cdot\,^{Y(e_{-i},\,v\eps{i},\,a)}\llbracket X(e_i,\,u,\,0),\,Y(e_{-i},\,e_jb\eps{i},\,0)]=\\
=\llbracket X(e_i,\,u,\,0),\,Y(e_{-i},\,v\eps{i},\,a)]\cdot\,^{Y(e_{-i},\,v\eps{i},\,a)}Y_{(i)}(u,\,e_jb,\,0)=\\
=\llbracket X(e_i,\,u,\,0),\,Y(e_{-i},\,v\eps{i},\,a)]\cdot\,^{Y(e_{-i},\,v\eps{i},\,a)}Y(e_j,\,ub,\,0).
\end{multline*}
Observe that in general $X(e_{-i},\,v\eps{i},\,a)$ does not lie in $\ps j$, but $X(e_j,\,ub,\,0)$ always lies in $\ps i$. So that we can compute the conjugate as follows
\begin{multline*}
\!\,^{Y(e_{-i},\,v\eps{i},\,a)}Y(e_j,\,ub,\,0)=Y(e_j,\,ub,\,0)[Y(e_j,\,-ub,\,0),\,Y(e_{-i},\,v\eps{i},\,a)]=\\
=Y(e_j,\,ub,\,0)Y(e_{-i},\,T(e_j,\,-ub,\,0)v\eps{i},\,a)\cdot Y(e_{-i},\,-v\eps{i},\,-a)=\\
=Y(e_j,\,ub,\,0)Y(e_{-i},\,v\eps{i}-ub\lan e_j,\,v\eps i\ran,\,a)Y(e_{-i},\,-v\eps{i},\,-a)=\\
=Y(e_j,\,ub,\,0)Y(e_{-i},\,-ubv_{-j}\eps i\eps j,\,0).
\end{multline*}
Thus,
\begin{multline*}
Y_{(i)}(u,\,v+e_jb,\,a+v_{-j}b\eps j)=\llbracket X(e_i,\,u,\,0),\,Y(e_{-i},\,v\eps{i},\,a)]Y(e_j,\,ub,\,0)\cdot\\
\cdot Y(e_{-i},\,-ubv_{-j}\eps i\eps j,\,0)Y(e_{-i},\,u(a+v_{-j}b\eps{-j})\eps{-i},\,0)=\\
=\llbracket X(e_i,\,u,\,0),\,Y(e_{-i},\,v\eps{i},\,a)]Y(e_j,\,ub,\,0)Y(e_{-i},\,ua\eps{-i},\,0).
\end{multline*}
Finally, it remains to observe that $X(e_j,\,ub,\,0)\in\ps{-i}$ and thus one has $[Y(e_j,\,ub,\,0),\,Y(e_{-i},\,ua\eps{-i},\,0)]=1$ and
\begin{multline*}
Y_{(i)}(u,\,v+e_jb,\,a+v_{-j}b\eps j)=\llbracket X(e_i,\,u,\,0),\,Y(e_{-i},\,v\eps{i},\,a)]Y(e_{-i},\,ua\eps{-i},\,0)\cdot\\
\cdot Y(e_j,\,ub,\,0)=Y_{(i)}(u,\,v,\,a)Y_{(i)}(u,\,e_jb,\,0).
\end{multline*}
\end{proof}

\begin{df}
For $u\in V$ such that $u_i=u_{-i}=0$, $v\in I^{2l}$ such that $\lan u,\,v\ran=0$, and $a\in R$ such that $a-\lan v_-,\,v_+\ran\in\Gamma$ define
\begin{multline*}
Y_{(i)}(u,\,v,\,a)=\\=Y_{(i)}(u,v-e_iv_i-e_{-i}v_{-i},a-v_iv_{-i}\eps i)Y(e_i,\,uv_i,\,0)Y(e_{-i},\,uv_{-i},\,0).
\end{multline*}
\end{df}

\begin{rk}
Observe that the above definition coincides with the old one for $v$ with $v_i=v_{-i}=0$, so that we can use the same notation for the generator. Observe also that the right-hand side is well-defined.
\end{rk}

\begin{lm}
\label{w-conjugation}
For $g\in\ls i$, $u\in V$ such that $u_i=u_{-i}=0$, $v\in I^{2l}$ such that $\lan u,\,v\ran=0$, and $a\in R$ such that $a-\lan v_-,\,v_+\ran\in\Gamma$ one has
$$
g\,Y_{(i)}(u,\,v,\,a)g\inv=Y_{(i)}(\phi(g)u,\,\phi(g)v,\,a).
$$
\end{lm}

\begin{proof}
Since $g\in\ls i$ one gets 
$$
\big(\phi(g)v\big)_i=\lan\phi(g)v,\,e_{-i}\ran\eps i=\lan\phi(g)v,\,\phi(g)e_{-i}\ran\eps i=\lan v,\,e_{-i}\ran\eps i=v_i
$$
and similarly $\big(\phi(g)v\big)_{-i}=v_{-i}$. Then
\begin{multline*}
\!\,^gY_{(i)}(u,\,v,\,a)=\\
=\!\,^gY_{(i)}(u,v-e_iv_i-e_{-i}v_{-i},a-v_iv_{-i}\eps i)\cdot\,^gY(e_i,\,uv_i,\,0)\cdot\,^gY(e_{-i},\,uv_{-i},\,0)=\\
=Y_{(i)}(\phi(g)u,\phi(g)v-\phi(g)e_iv_i-\phi(g)e_{-i}v_{-i},a-v_iv_{-i}\eps i)\cdot\\
\cdot Y(e_i,\,\phi(g)uv_i,\,0)Y(e_{-i},\,\phi(g)uv_{-i},\,0)=\\
=Y_{(i)}(\phi(g)u,\phi(g)v-e_i\big(\phi(g)v\big)_i-e_{-i}\big(\phi(g)v\big)_{-i},a-\big(\phi(g)v\big)_i\big(\phi(g)v\big)_{-i}\eps i)\cdot\\
\cdot Y(e_i,\,\phi(g)u\big(\phi(g)v\big)_i,\,0)Y(e_{-i},\,\phi(g)u\big(\phi(g)v\big)_{-i},\,0)=\\
=Y_{(i)}(\phi(g)u,\,\phi(g)v,\,a).
\end{multline*}
\end{proof}

\begin{rk}
Obviously, in the absolute case $(I,\,\Gamma)=(R,\,R)$ for $u$, $v\in V$ such that $u_i=u_{-i}=u_j=u_{-j}=0$ and $v_j=v_{-j}=0$, $j\not\in\{\pm i\}$, $a\in R$, one has that such an element $X_{(i)}(u,\,v,\,a)$ lies in $\ls j$. Indeed, it is a product of elements from $\ls j$ by definition.
\end{rk}

\begin{lm}
\label{w-correctness}
For $j\not\in\{\pm i\}$, $u\in V$ such that $u_i=u_{-i}=u_j=u_{-j}=0$, $v\in I^{2l}$ such that $\lan u,\,v\ran=0$, and $a\in R$ such that $a-\lan v_-,\,v_+\ran\in\Gamma$ one has
$$
Y_{(i)}(u,\,v,\,a)=Y_{(j)}(u,\,v,\,a).
$$
\end{lm}

\begin{proof}
Denote $\tilde v=v-e_iv_i-e_{-i}v_{-i}$, $\tilde a=a-v_iv_{-i}\eps i$. Then by Lemma~\ref{z-additivity} one gets
\begin{multline*}
Y_{(i)}(u,\,\tilde v,\,\tilde a)=Y_{(i)}(u,\,\tilde v-e_{-j}v_{-j},\,\tilde a-v_jv_{-j}\eps j)Y_{(i)}(u,\,e_{-j}v_{-j},\,0)=\\
=Y_{(i)}(u,\,\tilde v-e_{-j}v_{-j}-e_jv_j,\,\tilde a-v_jv_{-j}\eps j)Y_{(i)}(u,\,e_{j}v_{j},\,0)Y_{(i)}(u,\,e_{-j}v_{-j},\,0).
\end{multline*}
Further, denote $\tilde{\tilde v}=\tilde v-e_jv_j-e_{-j}v_{-j}$ and $\tilde{\tilde a}=\tilde a-v_jv_{-j}\eps j$. Then one has
\begin{multline*}
Y_{(i)}(u,\,v,\,a)=Y_{(i)}(u,\,\tilde{\tilde v},\,\tilde{\tilde a})Y(e_j,\,uv_j,\,0)Y(e_{-j},\,uv_{-j},\,0)\cdot\\
\cdot Y(e_i,\,uv_i,\,0)Y(e_{-i},\,uv_{-i},\,0).
\end{multline*}
Changing roles of $i$ and $j$ one gets
\begin{multline*}
Y_{(j)}(u,\,v,\,a)=Y_{(j)}(u,\,\tilde{\tilde v},\,\tilde{\tilde a})Y(e_i,\,uv_i,\,0)Y(e_{-i},\,uv_{-i},\,0)\cdot\\
\cdot Y(e_j,\,uv_j,\,0)Y(e_{-j},\,uv_{-j},\,0).
\end{multline*}
But $Y_{(i)}(u,\,\tilde{\tilde v},\,\tilde{\tilde a})=Y_{(j)}(u,\,\tilde{\tilde v},\,\tilde{\tilde a})$ by Lemma~\ref{z-correctness}. Finally, it remains to observe that $Y(e_{-i},\,uv_{-i},\,0)$ and $Y(e_i,\,uv_i,\,0)$ commute with both $Y(e_j,\,uv_j,\,0)$ and $Y(e_{-j},\,uv_{-j},\,0)$. This is obvious from the fact that $X(e_{-i},\,uv_{-i},\,0)$ and $X(e_i,\,uv_i,\,0)$ lie in $\ls j$.
\end{proof}

\begin{rk}
For $u$ equal to the base vector $e_j$ using Lemma~\ref{z-decomposition-for-y} one gets 
\begin{multline*}
Y_{(i)}(e_j,\,v,\,a)=\\=Y_{(i)}(e_j,v-e_iv_i-e_{-i}v_{-i},a-v_iv_{-i}\eps i)Y(e_i,\,e_jv_i,\,0)Y(e_{-i},\,e_jv_{-i},\,0)=\\=Y(e_j,v-e_iv_i-e_{-i}v_{-i},a-v_iv_{-i}\eps i)Y(e_j,\,e_iv_i,\,0)Y(e_{j},\,e_{-i}v_{-i},\,0)=\\=Y(e_j,\,v,\,a).
\end{multline*}
\end{rk}

\begin{df}
For $u$ having at least two pairs of zeros the element $Y_{(i)}(u,\,v,\,a)$ does not depend on the choice of $i$ by Lemma~\ref{w-correctness}. In this situation we will often omit the index in the notation,
$$
Y(u,\,v,\,a)=Y_{(i)}(u,\,v,\,a).
$$
\end{df}

\begin{df}
For $u\in V$, $w\in I^{2n}$, such that $\lan u,\,w\ran=0$, $w_i=w_{-i}=0$ and $a\in\lan w_-,\,w_+\ran+\Gamma$ define
\begin{multline*}
Z_{(i)}(u,\,w,\,a)=Y_{(i)}(u-e_iu_i-e_{-i}u_{-i},\,w,\,a)Y(e_iu_i+e_{-i}u_{-i},\,w,\,a)\cdot\\
\cdot Y(e_iu_i+e_{-i}u_{-i},\,(u-e_iu_i-e_{-i}u_{-i})a,\,0).
\end{multline*}
\end{df}

Our next objective is to show that $Z_{(i)}(u,\,0,\,a)$ does not depend on the choice of $i$.

\begin{lm}
\label{long-add}
For $u\in V$ such that $u_i=u_{-i}=u_j=u_{-j}=0$, $v\in I^{2l}$ such that $v_{i}=v_{-i}=0$, $\lan u,\,v\ran=0$, $a\in R$ such that $a-\lan v_-,\,v_+\ran\in\Gamma$, $b\in\Gamma$ one has
$$
Y(u,\,v,\,a+b)=Y(u,\,v,\,a)Y(u,\,0,\,b).
$$
\end{lm}

\begin{proof}
Decompose $Y(u,\,v,\,a+b)$ as follows
\begin{multline*}
Y(u,\,v,\,a+b)=Y_{(i)}(u,\,v,\,a+b)=\\
=\llbracket X(e_i,\,u,\,0),\,Y(e_{-i},\,v\eps i,\,a+b)]Y(e_{-i},\,u(a+b)\eps{-i},\,0).
\end{multline*}
Using $\llbracket a,\,bc]=\llbracket a,\,b]\cdot\,^b\llbracket a,\,c]$ we obtain
\begin{multline*}
\llbracket X(e_i,\,u,\,0),\,Y(e_{-i},\,v\eps i,\,a+b)]=\\
=\llbracket X(e_i,\,u,\,0),\,Y(e_{-i},\,v\eps i,\,a)Y(e_{-i},\,0,\,b)]=\\
=\llbracket X(e_i,\,u,\,0),\,Y(e_{-i},\,v\eps i,\,a)]\cdot\,^{Y(e_{-i},\,v\eps i,\,a)}\llbracket X(e_i,\,u,\,0),\,Y(e_{-i},\,0,\,b)].
\end{multline*}
Since $\lan u,\,v\ran=0$ one has 
$$\!\,^{X(e_{-i},\,v\eps i,\,a)}Y(e_{-i},\,ub\eps{-i},\,0)=Y(e_{-i},\,ub\eps{-i},\,0),$$ 
and thus
\begin{multline*}
Y(u,\,v,\,a+b)=\\
=\llbracket X(e_i,\,u,\,0),\,Y(e_{-i},\,v\eps i,\,a+b)]Y(e_{-i},\,ub\eps{-i},\,0)Y(e_{-i},\,ua\eps{-i},\,0)=\\
=\llbracket X(e_i,\,u,\,0),\,Y(e_{-i},\,v\eps i,\,a)]\cdot\,^{Y(e_{-i},\,v\eps i,\,a)}\llbracket X(e_i,\,u,\,0),\,Y(e_{-i},\,0,\,b)]\cdot\\
\cdot\,^{X(e_{-i},\,v\eps i,\,a)}Y(e_{-i},\,ub\eps{-i},\,0)\cdot Y(e_{-i},\,ua\eps{-i},\,0)=\\
=\llbracket X(e_i,\,u,\,0),\,Y(e_{-i},\,v\eps i,\,0)]\cdot\,^{Y(e_{-i},\,v\eps i,\,a)}Y_{(i)}(u,\,0,\,b)\cdot Y(e_{-i},\,ua\eps{-i},\,0).
\end{multline*}
Recall that $X_{(j)}(u,\,0,\,b)\in\ls i$ acts trivially on both $Y(e_{-i},\,v\eps i,\,a)$ and $Y(e_{-i},\,ua\eps{-i},\,0)$, so that
\begin{multline*}
Y(u,\,v,\,a+b)=[Y(e_i,\,u,\,0),\,Y(e_{-i},\,v\eps i,\,a)]Y(e_{-i},\,ua\eps{-i},\,0)Y(u,\,0,\,b)=\\
=Y_{(i)}(u,\,v,\,a)Y(u,\,0,\,b).
\end{multline*}
\end{proof}

\begin{rk}
For $u\in V$ having at least two pairs of zeros one has $Y(u,\,0,\,0)=1$ and $Y(u,\,0,\,a)\inv=Y(u,\,0,\,-a)$.
\end{rk}

\begin{lm}
\label{w=zz}
For $u\in V$ such that $u_i=u_{-i}=u_j=u_{-j}=0$, $v\in I^{2l}$ such that $\lan u,\,v\ran=0$ and $v_iv_{-i}\in\Gamma$, $a\in R$ such that $a-\lan v_-,\,v_+\ran\in\Gamma$ one has
$$
Y(u,\,v,\,a)=Y(u,\,v-e_iv_i-e_{-i}v_{-i},\,a)Y(u,\,e_iv_i+e_{-i}v_{-i},\,0).
$$
\end{lm}

\begin{proof}
By definition
\begin{multline*}
Y(u,\,v,\,a)=Y_{(i)}(u,\,v,\,a)=\\
=Y_{(i)}(u,v-e_iv_i-e_{-i}v_{-i},a-v_iv_{-i}\eps i)Y(e_i,\,uv_i,\,0)Y(e_{-i},\,uv_{-i},\,0).
\end{multline*}
Denote $\tilde v=v-e_iv_i-e_{-i}v_{-i}$, then by the previous lemma
$$
Y(u,\tilde v,a-v_iv_{-i}\eps i)=Y(u,\,\tilde v,\,a)Y(u,\,0,\,-v_iv_{-i}\eps i),
$$
and thus
\begin{multline*}
Y(u,\,v,\,a)=Y(u,\,\tilde v,\,a)Y(u,\,0,\,-v_iv_{-i}\eps i)Y(e_i,\,uv_i,\,0)Y(e_{-i},\,uv_{-i},\,0)=\\
=Y(u,\,\tilde v,\,a)Y_{(i)}(u,\,e_iv_i+e_{-i}v_{-i},\,0).
\end{multline*}
\end{proof}

\begin{cl*}
Consider $u\in V$, $w\in I^{2n}$, such that $\lan u,\,w\ran=0$, $w_i=w_{-i}=0$, $a\in\lan w_-,\,w_+\ran+\Gamma$, $j\not\in\{\pm i\}$ and denote $v=e_iu_i+e_{-i}u_{-i}$, $v'=e_ju_j+e_{-j}u_{-j}$, $\tilde u=u-v$, $\tilde{\tilde u}=\tilde u-v'$. Then one has
\begin{multline*}
Z_{(i)}(u,\,w,\,a)=Y_{(i)}(\tilde u,\,w,\,a)Y(v,\,w,\,a)Y(v,\,\tilde ua,\,0)=\\
=Y_{(i)}(\tilde u,\,w,\,a)Y(v,\,w,\,a)Y(v,\,\tilde{\tilde u}a,\,0)Y(v,\,v'a,\,0).
\end{multline*}
\end{cl*}

\begin{proof}
Since $l\geq3$ the vector $v$ has at least two pairs of zeros, so that one can apply the previous lemma.
\end{proof}

\begin{lm}
\label{short-is-three-long}
For $j$, $k\not\in\{\pm i\}$, $u$, $v\in V$ such that $u_i=u_{-i}=u_j=u_{-j}=0$ and $v_i=v_{-i}=v_k=v_{-k}=0$, $\lan u,\,v\ran=0$, $w\in I^{2l}$ such that $w_i=w_{-i}=0$, $\lan u,\,w\ran=\lan v,\,w\ran=0$ and $a\in\lan w_-,\,w_+\ran+\Gamma$ holds
$$
Y_{(i)}(u+v,\,w,\,a)=Y(u,\,w,\,a)\cdot Y(v,\,w,\,a)\cdot Y(v,\,ua,\,0).
$$
\end{lm}

\begin{proof}
Decomposing $X(e_i,\,u+v,\,0)=X(e_i,\,v,\,0)X(e_i,\,u,\,0)$ inside the commutator and using $\llbracket ab,\,c]=\,^a\llbracket b,\,c]\cdot\llbracket a,\,c]$ we get
\begin{multline*}
Y_{(i)}(u+v,\,w,\,a)=\llbracket X(e_i,\,u+v,\,0),\,Y(e_{-i},\,w,\,a)]Y(e_{-i},\,(u+v)a\eps{-i},\,0)=\\
=\llbracket X(e_i,\,v,\,0)X(e_i,\,u,\,0),\,Y(e_{-i},\,w,\,a)]Y(e_{-i},\,va\eps{-i},\,0)Y(e_{-i},\,ua\eps{-i},\,0)=\\
=\!\,^{X(e_i,\,v,\,0)}\llbracket X(e_i,\,u,\,0),\,Y(e_{-i},\,w,\,a)]\cdot\llbracket X(e_i,\,v,\,0),\,Y(e_{-i},\,w,\,a)]\cdot\\
\cdot Y(e_{-i},\,va\eps{-i},\,0)Y(e_{-i},\,ua\eps{-i},\,0)=\!\,^{X(e_i,\,v,\,0)}\llbracket X(e_i,\,u,\,0),\,Y(e_{-i},\,w,\,a)]\cdot\\
\cdot Y_{(i)}(v,\,w,\,a)Y(e_{-i},\,ua\eps{-i},\,0).
\end{multline*}
Observe that $Y_{(k)}(v,\,w,\,a)\in\ls i$ commutes with $Y(e_{-i},\,-ua\eps{-i},\,0)$, moreover, $X(e_i,\,v,\,0)\in\ls k$ and acts trivially on $Y_{(k)}(v,\,w,\,a)$, thus $Y_{(k)}(v,\,w,\,a)$ commutes with $\!\,^{X(e_i,\,v,\,0)}Y(e_{-i},\,-ua\eps{-i},\,0)$. Thus, we get
\begin{multline*}
\!\,^{X(e_i,\,v,\,0)}\llbracket X(e_i,\,u,\,0),\,Y(e_{-i},\,w,\,a)]\cdot Y(v,\,w,\,a)\cdot Y(e_{-i},\,ua\eps{-i},\,0)=\\
=\!\,^{X(e_i,\,v,\,0)}\llbracket X(e_i,\,u,\,0),\,Y(e_{-i},\,w,\,a)]\cdot\,^{X(e_i,\,v,\,0)}Y(e_{-i},\,ua\eps{-i},\,0)\cdot\\
\cdot\,^{X(e_i,\,v,\,0)}Y(e_{-i},\,-ua\eps{-i},\,0)\cdot Y(v,\,w,\,a)\cdot Y(e_{-i},\,ua\eps{-i},\,0)=\\
=\!\,^{X(e_i,\,v,\,0)}Y_{(i)}(u,\,w,\,a)\cdot Y(v,\,w,\,a)\cdot\llbracket X(e_i,\,v,\,0),\,Y(e_{-i},\,-ua\eps{-i},\,0)]=\\
=\!\,^{X(e_i,\,v,\,0)}Y(u,\,w,\,a)\cdot Y(v,\,w,\,a)\cdot Y(v,\,ua,\,0).
\end{multline*}
It only remains to show that $$\!\,^{X(e_i,\,v,\,0)}Y(u,\,w,\,a)=Y(u,\,w,\,a).$$ Decompose 
$$Y_{(-i)}(u,\,w,\,a)=\llbracket X(e_{-i},\,u,\,0),\,Y(e_i,\,w\eps{-i},\,a)]Y(e_i,\,ua\eps i,\,0),$$
then,
\begin{multline*}
\!\,^{X(e_i,\,v,\,0)}Y(u,\,w,\,a)=\\
=\!\,^{X(e_i,\,v,\,0)}\llbracket X(e_{-i},\,u,\,0),\,Y(e_i,\,w\eps{-i},\,a)]Y(e_i,\,ua\eps i,\,0)=\\
=\llbracket X(e_{-i}+v\eps i,\,u,\,0),\,Y(e_i,\,w\eps{-i},\,a)]Y(e_i,\,ua\eps i,\,0)=\\
=\llbracket X(e_{-i},\,u,\,0)X(v\eps i,\,u,\,0),\,Y(e_i,\,w\eps{-i},\,a)]Y(e_i,\,ua\eps i,\,0)=\\
=\!\,^{X(e_{-i},\,u,\,0)}\llbracket X(v\eps i,\,u,\,0),\,Y(e_i,\,w\eps{-i},\,a)]\cdot\\
\cdot\llbracket X(e_{-i},\,u,\,0),\,Y(e_i,\,w\eps{-i},\,a)]Y(e_i,\,ua\eps i,\,0)=\\
=\!\,^{X(e_{-i},\,u,\,0)}\llbracket X(v\eps i,\,u,\,0),\,Y(e_i,\,w\eps{-i},\,a)]Y(u,\,w,\,a).
\end{multline*}
Finally, notice that $X_{(k)}(v\eps i,\,u,\,0)\in\ls i$ commutes with $Y(e_i,\,w\eps{-i},\,a)$.
\end{proof}

\section{Case of maximal form parameter}

In this section we assume that $\Gamma=I$.

\begin{lm}
\label{short-symmetry}
For $k$, $j\not\in\{\pm i\}$, $u$, $v\in V$ such that $u_i=u_{-i}=u_{j}=u_{-j}=0$ and $v_i=v_{-i}=v_{k}=v_{-k}=0$, $\lan u,\,v\ran=0$, and $a\in I$, one has
$$
Y(u,\,va,\,0)=Y(v,\,ua,\,0).
$$
\end{lm}

\begin{proof}
By Lemma~\ref{short-is-three-long} one has
$$
Y(v,\,ua,\,0)=Y(v,\,0,\,-a)Y(u,\,0,\,-a)Y_{(i)}(u+v,\,0,\,a)
$$
and similarly
$$
Y(u,\,va,\,0)=Y(u,\,0,\,-a)Y(v,\,0,\,-a)Y_{(i)}(v+u,\,0,\,a).
$$
But $X_{(j)}(u,\,0,\,-a)\in\ls i$ acts trivially on $Y_{(i)}(v,\,0,\,-a)$.
\end{proof}

\begin{lm}
\label{correctness}
For $u\in V$, $a\in I$ and $j\not\in\{\pm i\}$
$$
Z_{(i)}(u,\,0,\,a)=Z_{(j)}(u,\,0,\,a).
$$
\end{lm}

\begin{proof}
Set 
$$v=e_iu_i+e_{-i}u_{-i},\quad v'=e_ju_j+e_{-j}u_{-j},\quad \tilde u=u-v,\quad \tilde{\tilde u}=\tilde u-v'.$$ 
Lemmas~\ref{w=zz} and \ref{short-is-three-long} imply that
\begin{multline*}
Z_{(i)}(u,\,0,\,a)=Y_{(i)}(\tilde u,\,0,\,a)Y(v,\,0,\,a)Y(v,\,\tilde ua,\,0)=\\
=Y_{(i)}(\tilde u,\,0,\,a)Y(v,\,0,\,a)Y(v,\,\tilde{\tilde u}a,\,0)Y(v,\,v'a,\,0)=\\
=Y(\tilde{\tilde u},\,0,\,a)Y(v',\,0,\,a)Y(v',\,\tilde{\tilde u}a,\,0)Y(v,\,0,\,a)Y(v,\,\tilde{\tilde u}a,\,0)Y(v,\,v'a,\,0).
\end{multline*}
Interchanging roles of $i$ and $j$ one has
\begin{multline*}
Z_{(j)}(u,\,0,\,a)=\\
=Y(\tilde{\tilde u},\,0,\,a)Y(v,\,0,\,a)Y(v,\,\tilde{\tilde u}a,\,0)Y(v',\,0,\,a)Y(v',\,\tilde{\tilde u}a,\,0)Y(v',\,va,\,0).
\end{multline*}
By Lemma~\ref{short-symmetry} one has $Y(v,\,v'a,\,0)=Y(v',\,va,\,0)$. Now we have to show that $Y(v,\,0,\,a)Y(v,\,\tilde{\tilde u}a,\,0)$ commutes with $Y(v',\,0,\,a)Y(v',\,\tilde{\tilde u}a,\,0)$. With this end fix $k\not\in\{\pm i,\pm j\}$. Observe that $Y_{(i)}(v',\,0,\,a)\in\ls k$ commutes with $Y_{(k)}(v,\,\tilde{\tilde u}a,\,0)$ by Lemma~\ref{w-conjugation}. Next, we will show that
$$
[Y(v,\,\tilde{\tilde u}a,\,0),\,Y(v',\,\tilde{\tilde u}a,\,0)]=1.
$$
By Lemma~\ref{short-symmetry} one has
$$
Y(v,\,\tilde{\tilde u}a,\,0)=Y(\tilde{\tilde u},\,va,\,0)
$$
and
$$
Y(v',\,\tilde{\tilde{u}}a,\,0)=Y(\tilde{\tilde{u}},\,v'a,\,0).
$$
Now, Lemma~\ref{w=zz} implies that
\begin{multline*}
Y(v,\,{\tilde{\tilde u}}a,\,0)Y(v',\,{\tilde{\tilde u}}a,\,0)=Y({\tilde{\tilde u}},\,va,\,0)Y({\tilde{\tilde u}},\,v'a,\,0)=\\
=Y(\tilde{{\tilde u}},\,va+v'a,\,0)=Y(\tilde{{\tilde u}},\,v'a,\,0)Y(\tilde{\tilde{u}},\,va,\,0)=\\
=Y(v',\,\tilde{\tilde{u}}a,\,0)Y(v,\,\tilde{{\tilde u}}a,\,0).
\end{multline*}
Finally, it remains to obseve that $Y_{(j)}(v,\,0,\,a)\in\ls k$ commutes with both $Y_{(k)}(v',\,0,\,a)$ and $Y_{(k)}(v',\,\tilde{\tilde u}a,\,0)$ by Lemma~\ref{w-conjugation}.
\end{proof}

\begin{rk}
Since $Z_{(i)}(u,\,0,\,a)$ does not depend on the choice of $i$ we will often omit the index in the notation
$$
Z(u,\,0,\,a)=Z_{(i)}(u,\,0,\,a).
$$
\end{rk}

Our next objective is to prove the following formula describing the action of $\St\!\Sp_{2l}(R)$ on the long-root type generators, namely
$$\!\,^gZ(u,\,0,\,a)=Z(\phi(g)u,\,0,\,a).$$ 

\begin{lm}
\label{conjugation-by-long}
For any $u\in V$, any $a\in\Gamma$, $b\in R$ and any index $i$, one has
$$
\!\,^{X_{i,-i}(b)}Z(u,\,0,\,a)=Z(T_{i,-i}(b)u,\,0,\,a).
$$
\end{lm}

\begin{proof}
The proof of Lemma~26 from the Another Presentation paper can be repeated verbatim.
\end{proof}

\begin{lm}
\label{lm0}
For $u\in V$, $v$, $w\in I^{2l}$ and $j\not\in\{\pm i\}$ such that 
$$u_i=u_{-i}=u_j=u_{-j}=0,\qquad v_j=v_{-j}=0,\qquad w_j=w_{-j}=0,$$
 and $\lan u,\,v\ran=0$, $\lan u,\,w\ran=0$, $\lan v,\,w\ran=0$, one has
$$
Y(u,\,v,\,0)Y(u,\,w,\,0)=Y(u,\,v+w,\,0).
$$
\end{lm}

\begin{proof}
The proof of Lemma~27 from the Another Presentation paper can be repeated verbatim.
\end{proof}

\begin{lm}
\label{w-symmetry}
Let $\Card\{\pm i,\pm j,\pm k\}=6$ and let $v\in I^{2l}$, $v'\in V$ be vectors having only $\pm i$-th and $\pm j$-th non-zero coordinates respectively; consider also $v''\in V$ such that $(v'')_i=(v'')_{-i}=(v'')_j=(v'')_{-j}=0$. Set $w=v'+v''$. Then
$$Y(v'',\,v,\,0)Y(v',\,v,\,0)=Y_{(i)}(w,\,v,\,0).$$
\end{lm}

\begin{proof}
The proof of Lemma~28 from the Another Presentation paper can be repeated verbatim.
\end{proof}

\begin{lm}
\label{conjugation-by-short}
For any $j\not\in\{\pm k\}$, any $u\in V$ and any $a\in I$, $b\in R$, one has
$$
\!\,^{X_{jk}(b)}Z(u,\,0,\,a)=Z(T_{jk}(b)u,\,0,\,a).
$$
\end{lm}

\begin{proof}
The proof of Lemma~29 from the Another presentation paper can be repeated almost verbatim. The only difference is that one should use another approach to show that $X_{jk}(b)$ acts trivially on $Y(v,\,0,\,a)$ and $Y(v,\,\tilde{\tilde{\tilde u}}a,\,0)$. Namely, one should observe first that $X_{jk}(b)\in\ls i$ acts trivially on $$Y_{(i)}(\tilde{\tilde{\tilde u}}a,\,v,\,0)=Y(v,\,\tilde{\tilde{\tilde u}}a,\,0).$$
Afterwards, decompose
$$
Y(v,\,0,\,a)=\llbracket X(e_{-j},\,v,\,0),\,Y(e_j,\,0,\,a)]Y(e_j,\,va\eps j,\,0),
$$
thus,
$$
\!\,^{X_{jk(b)}}Y(v,\,0,\,a)=\llbracket X(T_{jk}(b)e_{-j},\,v,\,0),\,Y(e_j,\,0,\,a)]Y(e_j,\,va\eps j,\,0).
$$
Since 
\begin{multline*}X(T_{jk}(b)e_{-j},\,v,\,0)=X(e_{-j}-e_{-k}a\eps k\eps j,\,v,\,0)=\\=X(e_{-j},\,v,\,0)X(-e_{-k}a\eps k\eps j,\,v,\,0)=X(e_{-j},\,v,\,0)X(e_{-k},\,-va\eps k\eps j,\,0),
\end{multline*}
we obtain that
\begin{multline*}
\!\,^{X_{jk(b)}}Y(v,\,0,\,a)=\\
=\llbracket X(e_{-j},\,v,\,0)X(e_{-k},\,-va\eps k\eps j,\,0),\,Y(e_j,\,0,\,a)]Y(e_j,\,va\eps j,\,0)=\\
=\!\,^{X(e_{-j},\,v,\,0)}\llbracket X(e_{-k},\,-va\eps k\eps j,\,0),\,Y(e_j,\,0,\,a)]\cdot Y(v,\,0,\,a).
\end{multline*}
Finally, it remains to observe that $X(e_{-j},\,v,\,0)\in\ls j$ acts trivially on $Y(e_j,\,0,\,a)$.
\end{proof}

\begin{cl*}
Lemmas~{\rm \ref{conjugation-by-long}} and {\rm \ref{conjugation-by-short}} imply that for any $g\in\St\!\Sp(2l,\,R)$, one has
$$
\!\,^gZ(u,\,0,\,a)=Z(\phi(g)u,\,0,\,a).
$$
\end{cl*}

\begin{lm}
\label{generating}
The set of elements $\{Z(u,\,0,\,a)\mid u\in V,\ a\in I\}$ generates $\St\!\Sp_{2l}(R,\,I)$ as a group.
\end{lm}

\begin{proof}
Firstly, choosing some $i$ and $j$ such that $\Card\{\pm i,\pm j\}=4$ one has
$$
Y_{(j)}(e_i,\,0,\,a)=Y(e_i,\,0,\,a)Y(0,\,0,\,a)Y(0,\,e_ia,\,0)=Y(e_i,\,0,\,a)=Z_{i,-i}(a).
$$
Now, choosing $k\not\in\{\pm i,\pm j\}$, taking $u=e_{-k}$, $v=-e_j\eps k$ and any $a\in I$, and using Lemma~\ref{short-is-three-long} we obtain
\begin{multline*}
X_{jk}(a)=Y(u,\,va,\,0)=\\
=Y(-e_j\eps k,\,0,\,-a)Y(e_{-k},\,0,\,-a)Y_{(i)}(e_{-k}-e_j\eps k,\,0,\,a)=\\
=Z_{(i)}(-e_j\eps k,\,0,\,-a)Z_{(i)}(e_{-k},\,0,\,-a)Z_{(i)}(e_{-k}-e_j\eps k,\,0,\,a).
\end{multline*}
\end{proof}

\begin{cl*}
Clearly, $\Ker\phi$ acts trivially on $\St\!\Sp_{2l}(R,\,I)$.
\end{cl*}

\begin{lm}
\label{long-additivity}
For $u\in V$, $a$, $b\in I$ one has
$$
Z(u,\,0,\,a)Z(u,\,0,\,b)=Z(u,\,0,\,a+b).
$$
\end{lm}

\begin{proof}
The proof of Lemma~31 from the Another Presentation paper can be repeated verbatim.
\end{proof}

\begin{lm}
\label{long-scalar}
For $j\not\in\{\pm i\}$, $u\in V$ such that $u_i=u_{-i}=u_j=u_{-j}=0$ and $a\in I$, $b\in R$, one has
$$
Y(ub,\,0,\,a)=Y(u,\,0,\,ab^2).
$$
\end{lm}

\begin{proof}
The claim follows directly from Lemma~16.
\end{proof}

\begin{lm}
\label{x-long-scalar}
For $u\in V$, $a\in I$, $b\in R$, one has
$$
Z(ub,\,0,\,a)=Z(u,\,0,\,ab^2).
$$
\end{lm}

\begin{proof}
The proof of Lemma~33 from the Another Presentation paper can be repeated verbatim.\end{proof}

\begin{df}
For $u$, $v\in V$ such that $\lan u,\,v\ran=0$, $a\in I$ set
$$
Z(v,\,u,\,a,\,0)=Z(v,\,0,\,-a)Z(u,\,0,\,-a)Z(u+v,\,0,\,a).
$$
\end{df}

\begin{lm}
\label{short-obvious}
For $g\in\St\!\Sp(2l,\,R)$ and $u$, $v\in V$ such that $\lan u,\,v\ran=0$, $a\in I$, $b\in R$, one has
\begin{enumerate}
\item $Z(v,\,u,\,a,\,0)=Z(u,\,v,\,a,\,0);$
\item $\!\,^gZ(u,\,v,\,a,\,0)=Z(\phi(g)u,\,\phi(g)v,\,a,\,0);$
\item $Z(u,\,ub,\,a,\,0)=Z(u,\,0,\,2ab)$.
\end{enumerate}
\end{lm}

\begin{proof}
Since {\it a}) and {\it b}) are obvious, it remains only to check {\it c}). By the very definition we have
$$
Z(u,\,ub,\,a,\,0)=Z(u,\,0,\,-a)Z(ub,\,0,\,-a)Z(ub+u,\,0,\,a).
$$
Then, using Lemma~\ref{x-long-scalar} and then Lemma~\ref{long-additivity}, we get
\begin{multline*}
Z(u,\,0,\,-a)Z(ub,\,0,\,-a)Z(ub+u,\,0,\,a)=\\
=Z(u,\,0,\,-a)Z(u,\,0,\,-ab^2)Z(u,\,0,\,a(b+1)^2)=Z(u,\,0,\,2ab).
\end{multline*}
\end{proof}

\begin{lm}
\label{new}
Consider $u$, $w\in V$ such that 
$$\lan u,\,w\ran=0,\qquad w_i=w_{-i}=w_j=w_{-j}=0,$$
where $i\not\in\{\pm j\}.$ Then
$$
Z(u+w,\,0,\,a)=Z(u,\,0,\,a)Z(w,\,0,\,a)Y(w,\,ua,\,0).
$$
\end{lm}

\begin{proof}
The proof of Lemma~35 from the Another Presentation paper can be repeated verbatim.
\end{proof}

\begin{lm}
\label{x=y}
For $v\in V$ such that $v_{-i}=0$, $a\in I$, one has
$$
Z(e_i,\,v,\,a,\,0)=Y(e_i,\,va,\,0).
$$
\end{lm}

\begin{proof}
In the statement of Lemma~\ref{new} take $u=v$, $w=e_i$.
\end{proof}

\begin{cl*}
Let $u$ be a column of a symplectic elementary matrix, $v$, $w\in V$ such that $\lan u,\,v\ran=\lan u,\,w\ran=0$, $a$, $b\in I$ such that $va=wb$. Then one has
$$
Z(u,\,v,\,a,\,0)=Z(u,\,w,\,b,\,0).
$$
\end{cl*}

\begin{proof}
Take $g\in\St\!\Sp_{2l}(R)$ such that $\phi(g)u=e_i$. Then one has
\begin{multline*}
Z(u,\,v,\,a,\,0)=\!\,^gZ(e_i,\,\phi(g)\inv v,\,a,\,0)=\!\,^gY(e_i,\,\phi(g)\inv va,\,0)=\\=\!\,^gY(e_i,\,\phi(g)\inv wb,\,0)=\!\,^gZ(e_i,\,\phi(g)\inv w,\,b,\,0)=Z(u,\,w,\,b,\,0).
\end{multline*}
\end{proof}

%
%

\begin{lm}
\label{x-long-is-three-short}
Consider $a\in I$, $r\in R$ and $u$, $v\in V$ such that $\lan u,\,v\ran=0$ and assume also that $v$ is a column of a symplectic elementary matrix. Then
$$
Z(u+vr,\,0,\,a)=Z(u,\,0,\,a)Z(v,\,0,\,ar^2)Z(v,\,u,\,ar,\,0).
$$
\end{lm}

\begin{proof}
Take $g\in\St\!\Sp_{2l}(R)$ such that $\phi(g)v=e_i$. Then,
$$
\!\,^gZ(u+vr,\,0,\,a)=Z(\phi(g)u+e_ir,\,0,\,a).
$$
Now, Lemma~\ref{new} (and Lemma~\ref{x=y}) imply that
\begin{multline*}
Z(\phi(g)u+e_ir,\,0,\,a)=Z(\phi(g)u,\,0,\,a)Z(e_ir,\,0,\,a)Y(e_ir,\,\phi(g)ua,\,0)=\\
=\!\,^g\Big(Z(u,\,0,\,a)Z(v,\,0,\,ar^2)Z(v,\,u,\,ar,\,0)\Big).
\end{multline*}
\end{proof}

\begin{lm}
\label{short-additivity}
Let $u\in V$ be column of a symplectic elementary matrix and let $v$, $w\in V$ be arbitrary coloumns such that $\lan u,\,v\ran=\lan u,\,w\ran=0$, $a$, $b\in I$. Then
\begin{align*}
&\text{a{\rm)} }Z(u,\,v,\,a,\,0)Z(u,\,w,\,a,\,0)=Z(u,\,v+w,\,a,\,0)Z(u,\,0,\,a^2\lan v,\,w\ran),\\
&\text{b{\rm)} }Z(u,\,v,\,a,\,0)Z(u,\,v,\,b,\,0)=Z(u,\,v,\,a+b,\,0).
\end{align*}
\end{lm}

\begin{proof}
Use the same trick as in Lemma~\ref{x-long-is-three-short}.
\end{proof}

\begin{df}
For $u$, $v\in V$ such that $\lan u,\,v\ran=0$, $a$, $b\in I$, set
$$
Z(u,\,v,\,a,\,b)=Z(u,\,v,\,a,\,0)Z(u,\,0,\,b).
$$
\end{df}

\begin{lm}
\label{p-relations}
Assume that $u$, $u'$ are columns of symplectic elementary matrices, and let $v$, $v'$, $w\in V$ be arbitrary coloumns such that $\lan u,\,v\ran=\lan u,\,w\ran=0$, and $\lan u',\,v'\ran=0$. Then for any $a$, $a'$, $b$, $b'\in I$, $r\in R$ one has
\begin{align*}
&\text{a{\rm)} }Z(u,\,vr,\,a,\,b)=Z(u,\,v,\,ar,\,b),\\
&\text{b{\rm)} }Z(u,\,v,\,a,\,b)Z(u,\,w,\,a,\,c)=Z(u,\,v+w,\,a,\,b+c+a^2\lan v,\,w\ran),\\
&\text{c{\rm)} }Z(u,\,v,\,a,\,0)Z(u,\,v,\,b,\,0)=Z(u,\,v,\,a+b,\,0),\\
&\text{d{\rm)} }Z(u,\,v,\,a,\,0)=Z(v,\,u,\,a,\,0),\\
&\begin{aligned}\text{e{\rm)} }Z(u',\,v',\,a',\,b')Z(u,\,v,\,a,\,b&)Z(u',\,v',\,a',\,b')\inv=\\&=Z(T(u',\,v'a',\,b')u,\,T(u',\,v'a',\,b')v,\,a,\,b),\end{aligned}\\
&\text{f{\rm)} }Z(u,\,u,\,a,\,0)=Z(u,\,0,\,0,\,2a),\\
&\text{g{\rm)} }Z(v+ur,\,0,\,0,\,a)=Z(v,\,0,\,0,\,a)Z(u,\,0,\,0,\,ar^2)Z(u,\,v,\,ar,\,0).
\end{align*}
\end{lm}

\begin{proof}
For {\it b}) use Lemmas~\ref{short-additivity} and \ref{long-additivity}, the rest was already checked.
\end{proof}

\begin{df}
Let the relative symplectic van der Kallen group $\St\!\Sp^*_{2l}(R,\,I)$ be the group defined by the set of generators 
\begin{multline*}
\big\{(u,\,v,\,a,\,b)\in V\times V\times I\times I\,\big|\,\text{$u$ is a column of }\\ \text{a symplectic elementary matrix},\ \lan u,\,v\ran=0\big\}
\end{multline*}
and relations
\setcounter{equation}{0}
\renewcommand{\theequation}{T\arabic{equation}}
\begin{align}
&(u,\,vr,\,a,\,b)=(u,\,v,\,ar,\,b)\text{ for any $r\in R$},\\
&(u,\,v,\,a,\,b)(u,\,w,\,a,\,c)=(u,\,v+w,\,a,\,b+c+a^2\lan v,\,w\ran),\\
&(u,\,v,\,a,\,0)(u,\,v,\,b,\,0)=(u,\,v,\,a+b,\,0),\\
&\begin{aligned}(u,\,v,\,a,\,0)=(v,\,u,\,a,\,0)\,\text{ for $v$ a co}&\text{lumn of}\\&\text{a symplectic elementary matrix,}\end{aligned}\\
&\begin{aligned}(u',\,v',\,a',\,b')(u,\,v,\,a,\,b)(u'&,\,v',\,a',\,b')\inv=\\&=(T(u',\,v'a',\,b')u,\,T(u',\,v'a',\,b')v,\,a,\,b),\end{aligned}\\
&(u,\,u,\,a,\,0)=(u,\,0,\,0,\,2a),\\
&\begin{aligned}(u+vr,\,0,\,0,\,a)=(u,\,0,\,0,\,a)(v,\,0,\,0,\,ar^2)(v,\,u,\,ar,\,0)\text{ for $v$, $u+vr$}&\\ \text{also columns of symplectic elementary matrices}.&\end{aligned}
\end{align}
\end{df}

\begin{rk}
Clearly, Lemma~\ref{p-relations} amounts to the existence of a homomorphism 
$$\varpi\colon\St\!\Sp^*_{2l}(R,\,I)\epi\St\!\Sp_{2l}(R,\,I),$$
sending $(u,\,v,\,a,\,b)$ to $Z(u,\,v,\,a,\,b)$.
\end{rk}

\begin{lm}
Any triple $(u,\,v,\,a)\in V\times V\times R$ defines a homomorphism $$\alpha_{u,v,a}\colon\St\!\Sp^*_{2l}(R,\,I)\rightarrow\St\!\Sp^*_{2l}(R,\,I)$$ sending generators $(u',\,v',\,a',\,b')$ to $(T(u,\,v,\,a)u',\,T(u,\,v,\,a)v',\,a',\,b')$.
\end{lm}

\begin{proof}
To show that $\alpha_{u,v,a}$ is well-defined we have to check that T1--T7 hold for the images of the generators. All of them are obvious, except T5, which is checked below. 
\begin{multline*}
(T(u,\,v,\,a)u',\,T(u,\,v,\,a)v',\,a',\,b')(T(u,\,v,\,a)u'',\,T(u,\,v,\,a)v'',\,a'',\,b'')\cdot\\
\cdot(T(u,\,v,\,a)u',\,T(u,\,v,\,a)v',\,a',\,b')\inv=\\
=(T(T(u,\,v,\,a)u',\,T(u,\,v,\,a)v'a',\,b')T(u,\,v,\,a)u'',\\
T(T(u,\,v,\,a)u',\,T(u,\,v,\,a)v'a',\,b')T(u,\,v,\,a)v'',\,a'',\,b'')=\\
=(T(u,\,v,\,a)T(u',\,v',\,a')u'',\,T(u,\,v,\,a)T(u',\,v',\,a')v'',\,a'',\,b'').
\end{multline*}
\end{proof}

\begin{lm}
There exists a well-defined homomorphism $$\St\!\Sp_{2l}(R)\rightarrow\mathrm{Aut}\,(\St\!\Sp^*_{2l}(R,\,I))$$ sending $X(u,\,v,\,a)$ to $\alpha_{u,v,a}$, i.e., absolute Steinberg group acts on relative van der Kallen group.
\end{lm}

\begin{proof}
We need to verify that P1--P3 hold for $\alpha_{u,v,a}$. We check P1 below, P2 and P3 are left to the reader.
\begin{multline*}
\alpha_{u,v,a}\alpha_{u,w,b}(u',\,v',\,a',\,b')=\alpha_{u,v,a}(T(u,\,w,\,b)u',\,T(u,\,w,\,b)v',\,a',\,b')=\\
=(T(u,\,v,\,a)T(u,\,w,\,b)u',\,T(u,\,v,\,a)T(u,\,w,\,b)v',\,a',\,b')=\\
=(T(u,\,v+w,\,a+b+\lan v,\,w\ran)u',\,T(u,\,v+w,\,a+b+\lan v,\,w\ran)v',\,a',\,b')=\\
=\alpha_{u,v+w,a+b+\lan v,\,w\ran}(u',\,v',\,a',\,b').
\end{multline*}
\end{proof}

\begin{rk}
Notice that $\varpi$ preserves the action of $\St\!\Sp_{2l}(R)$.
\end{rk}

\begin{df}
Set
\begin{align*}
&\!\,_{ij}(a)=(e_{-j},\,e_{i},\,a\eps{-j},\,0) \text{ for } i\not\in\{\pm j\},\\ 
&\!\,_{i,\,-i}(a)=(e_i,\,0,\,0,\,a).
\end{align*}
\end{df}

\begin{lm}
\label{st-relations}
Steinberg relations {\rm KL0}--{\rm KL7} hold for $\!\,_{ij}(a)$ and $\!\,_{i,-i}(a)$.
\end{lm}

\begin{proof}
To check KL0 use T4, for KL1 use T3 and T2. KL2 follows from the definition of the action. Below we verify the rest.
\begin{multline*}
{\rm (KL3)\ }\llbracket X_{ij}(r),\,\!\,_{jk}(b)]=\\
\phantom{\rm (KL3)\ }=(e_{-k},\,T(e_i,\,e_{-j}r\eps{-j},\,0)e_j,\,b\eps{-k},\,0)(e_{-k},\,e_j,\,b\eps{-k},\,0)\inv=\\
\phantom{\rm (KL3)\ }=(e_{-k},\,e_j+e_ir,\,b\eps{-k},\,0)(e_{-k},\,-e_j,\,b\eps{-k},\,0)=\\
\phantom{\rm (KL3)\ }=(e_{-k},\,e_ir,\,b\eps{-k},\,0)=(e_{-k},\,e_i,\,rb\eps{-k},\,0)=\\=(e_i,\,e_{-k},\,rb\eps{-k},\,0)=\!\,_{ik}(rb);
\end{multline*}
\begin{multline*}
{\rm (KL4)\ }\llbracket X_{i,-i}(r),\,\!\,_{-i,j}(b)]=\\
\phantom{\rm (KL3)\ }=(e_{-j},\,T(e_i,\,0,\,r)e_{-i},\,b\eps{-j},\,0)(e_{-j},\,e_{-i},\,b\eps{-j},\,0)\inv=\\
\phantom{\rm (KL3)\ }=(e_{-j},\,e_{-i}+e_ir\eps i,\,b\eps{-j},\,0)(e_{-j},\,-e_{-i},\,b\eps{-j},\,0)=\\
\phantom{\rm (KL3)\ }=(e_{-j},\,e_ir\eps i,\,b\eps{-j},\,-rb^2)=\\
\phantom{\rm (KL3)\ }=(e_{-j},\,e_ir\eps i,\,b\eps{-j},\,0)(e_{-j},\,0,\,b\eps{-j},\,-rb^2)=\\
\phantom{\rm (KL3)\ }=(e_{-j},\,e_i,\,rb\eps i\eps{-j},\,0)(e_{-j},\,0,\,0,\,-rb^2)=\!\,_{ij}(rb\eps i)\cdot\!\,_{-j,j}(-rb^2);
\end{multline*}
\begin{multline*}
{\rm (KL5)\ }[\,_{i,-i}(a),\,X_{-i,j}(s)\rrbracket=\\
\phantom{\rm (KL3)\ }=(e_i,\,0,\,0,\,a)(T(e_{-i},\,e_{-j}s\eps{-j},\,0)e_i,\,0,\,0,\,-a)=\\
\phantom{\rm (KL3)\ }=(e_i,\,0,\,0,\,a)(e_i+e_{-j}s\eps{-j}\eps{-i},\,0,\,0,\,-a)=\\
\phantom{\rm (KL3)\ }=(e_i,\,0,\,0,\,a)(e_i,\,0,\,0,\,-a)(e_{-j},\,0,\,0,\,-as^2)(e_{-j},\,e_i,\,-as\eps i\eps j,\,0)=\\
=(e_{-j},\,e_i,\,as\eps i\eps{-j},\,0)(e_{-j},\,0,\,0,\,-as^2)=\!\,_{ij}(as\eps i)\cdot\!\,_{-j,j}(-as^2);
\end{multline*}
\begin{multline*}
{\rm (KL6)\ }\llbracket X_{ij}(r),\,\!\,_{j,-i}(b)]=\\
\phantom{\rm (KL3)\ }=(e_i,\,T(e_i,\,e_{-j}r\eps{-j},\,0)e_j,\,b\eps i,\,0)(e_i,\,-e_j,\,b\eps i,\,0)=\\
\phantom{\rm (KL3)\ }=(e_i,\,e_j+e_ir,\,b\eps i,\,0)(e_i,\,-e_j,\,b\eps i,\,0)=(e_i,\,e_ir,\,b\eps i,\,0)=\\
=(e_i,\,e_i,\,rb\eps i,\,0)=(e_i,0,\,0,\,2rb\eps i)=\!\,_{i,-i}(2rb\eps i).
\end{multline*}
The definition of the action together with T5 imply KL7. 
\end{proof}

\begin{cl*}
There is a homomorphism $$\varrho\colon\St\!\Sp_{2l}(R,\,I)\rightarrow\St\!\Sp^*_{2l}(R,\,I)$$ sending $Y_{ij}(a)$ to $\!\,_{ij}(a)$ and preserving the action of $\St\!\Sp_{2l}(R)$. Obviously, $\varpi\varrho=1$, i.e. $\varrho$ is a splitting for $\varpi$.
\end{cl*}

\begin{lm}
\label{vdk-unipotent-decomposition}
For $v\in V$ such that $v_{-i}=0$, set 
$$\tilde v_-=\sum_{\stackrel{i\not\in\{\pm j\}}{i<0}}e_iv_i\qquad\text{and similarly}\qquad \tilde v_+=\sum_{\stackrel{i\not\in\{\pm j\}}{i>0}}e_iv_i.$$ 
Then
\begin{multline*}
(e_i,\,v,\,a,\,b)=\\
=\!\,_{i,-i}(b+2av_i-a^2\lan\tilde v_-,\,\tilde v_+\ran)\prod_{\stackrel{i\not\in\{\pm j\}}{i<0}}\!\,_{j,-i}(av_j\eps i)\prod_{\stackrel{i\not\in\{\pm j\}}{i>0}}\!\,_{j,-i}(av_j\eps i).
\end{multline*}
\end{lm}

\begin{proof}
\begin{multline*}
(e_i,\,v,\,a,\,b)=(e_i,\,0,\,a,\,b)(e_i,\,v,\,a,\,0)=\\
=(e_i,\,0,\,0,\,b)(e_i,\,e_iv_i,\,a,\,0)(e_i,\,\tilde v_-+\tilde v_+,\,a,\,0)=\\
=(e_i,\,0,\,0,\,b)(e_i,\,e_i,\,av_i,\,0)(e_i,\,0,\,a,\,-a^2\lan\tilde v_-,\,\tilde v_+\ran)\cdot\\
\cdot(e_i,\,\tilde v_-+\tilde v_+,\,a,\,a^2\lan\tilde v_-,\,\tilde v_+\ran)=\\
=(e_i,\,0,\,0,\,b)(e_i,\,0,\,0,\,2av_i)(e_i,\,0,\,0,\,-a^2\lan\tilde v_-,\,\tilde v_+\ran)\cdot\\
\cdot(e_i,\,\tilde v_-,\,a,\,0)(e_i,\,\tilde v_+,\,a,\,0)=\\
=(e_i,\,0,\,0,\,b+2av_i-a^2\lan\tilde v_-,\,\tilde v_+\ran)\cdot\\
\cdot\prod_{\stackrel{i\not\in\{\pm j\}}{i<0}}(e_i,\,e_jv_j,\,a,\,0)\prod_{\stackrel{i\not\in\{\pm j\}}{i>0}}(e_i,\,e_jv_j,\,a,\,0)=\\
=\!\,_{i,-i}(b+2av_i-a^2\lan\tilde v_-,\,\tilde v_+\ran)\prod_{\stackrel{i\not\in\{\pm j\}}{i<0}}\!\,_{j,-i}(av_j\eps i)\prod_{\stackrel{i\not\in\{\pm j\}}{i>0}}\!\,_{j,-i}(av_j\eps i).
\end{multline*}
\end{proof}

\begin{lm}
$$\varrho\varpi=1$$
\end{lm}

\begin{proof}
Obviously, $\varrho\varpi(\,_{ij}(a))=\!\,_{ij}(a)$, then $\varrho\varpi\big((e_i,\,v,\,a,\,b)\big)=(e_i,\,v,\,a,\,b)$ by the previous lemma. For arbitrary column of symplectic elementary matrix $u$ take $g\in\St\!\Sp_{2l}(R)$ such that $\phi(g)e_i=u$ and use that $\varrho$ and $\varpi$ preserve action.
\begin{multline*}
\varrho\varpi\big((u,\,v,\,a,\,b)\big)=\varrho\varpi\big(\!\,^g(e_i,\,\phi(g)\inv v,\,a,\,b)\big)=\\=\!\,^g\varrho\varpi\big((e_i,\,\phi(g)\inv v,\,a,\,b)\big)
=\!\,^g(e_i,\,\phi(g)\inv v,\,a,\,b)=(u,\,v,\,a,\,b).
\end{multline*} 
\end{proof}

\end{document}